\documentclass[12pt,notitlepage,a4paper,leqno]{article}
\usepackage{amsmath,amssymb,amsthm,amscd}
\usepackage[arrow,matrix]{xy}
\theoremstyle{plain}
\numberwithin{equation}{section}
\newtheorem{thm}{Theorem}[section]
\newtheorem{cor}[thm]{Corollary}
\newtheorem{lem}[thm]{Lemma}
\newtheorem{conj}{Conjecture}
\theoremstyle{definition}
\newtheorem{df}[thm]{Definition}
\newtheorem{rmk}[thm]{Remark}
\newcommand{\mb}{\mathbb}
\newcommand{\mf}{\mathfrak}
\newcommand{\ml}{\mathcal}
\newcommand{\CH}{{\rm CH}}
\newcommand{\T}{\boldsymbol{T}}
\begin{document}
\author{Rin Sugiyama\footnote{The author is supported by JSPS Research Fellowships for Young Scientists.}
}
\title{Tate conjecture for products of Fermat varieties over finite fields}
\date{}
\maketitle
\begin{abstract}
We prove under some assumptions that the Tate conjecture holds for products of Fermat varieties of different degrees. The method is to use a combinatorial property of eigenvalues of geometric Frobenius acting on $\ell$-adic \'etale cohomology.
\end{abstract}

\bigskip


\section{Introduction}
Let $p$ be a prime number. Let $k$ be a finite field of characteristic $p$. Let $\overline{k}$ be a separable closure of $k$ and let $G_k$ be the Galois group of $\overline{k}/k$. Let $\ell$ be a prime number different from $p$. For a projective smooth and geometrically integral variety $X$ over $k$, $\CH^i(X)$ denotes the Chow group of algebraic cycles on $X$ of codimension $i$ modulo rational equivalence. We first state the following famous conjecture raised by Tate \cite{Ta1}:
\begin{conj}[$\T(X/k,i)$]
The cycle class map
\begin{equation}\label{map}
\rho_X^i : \CH^i(X)\otimes \mb{Q}_{\ell}\longrightarrow H^{2i}(\overline{X},\mb{Q}_{\ell}(i))^{G_k}
\end{equation}
is surjective.
\end{conj}
Here $H^{2i}(\overline{X},\mb{Q}_{\ell}(i))^{G_k}$ denotes the $G_k$-invariant part of the $\ell$-adic cohomology $H^{2i}(\overline{X},\mb{Q}_{\ell}(i))$ (cf. Notation) of $\overline{X}:=X\times_k\overline{k}$. Tate \cite{Ta2} proved that $\T(X/k,1)$ holds in case where $X$ is an abelian variety or a product of curves. In case where $X$ is a  product of curves, Soul\'e \cite{So} proved that the Tate conjecture holds for $i=0,1,\dim X-1, \dim X$. Spiess \cite{Sp} proved that the Tate conjecture holds for products of elliptic curves. For other known cases we refer to \cite{Ta3}.

Beilinson \cite{Be} furthermore conjectured that $\rho_X^i$ is also injective. Therefore the map $\rho_X^i$ is conjectured to be bijective and we call the conjecture \textit{the Tate-Beilinson conjecture}. 

Let $K_i(X)$ be Quillen's higher $K$-groups associated to the category of vector bundles on $X$. Let $K_i(X)_{\mb{Q}}=K_i(X)\otimes \mb{Q}$. The Tate conjecture can be reformulated in terms of $K$-groups, because there is an  isomorphism
$$K_0(X)^{(i)}_{\mb{Q}}\simeq \CH^i(X)\otimes \mb{Q}.$$
Here $K_0(X)^{(i)}_{\mb{Q}}$ denotes the subspace of the rational $K$-group $K_0(X)_{\mb{Q}}$ on which the Adams operator $\psi^m$ acts as the multiplication by $m^i$ for all $m\in \mb{N}$. For higher $K$-groups of projective smooth varieties over finite fields, there is a conjecture raised by Parshin:
\begin{conj} For every projective smooth variety $X$ over a finite field and every integer $i>0$, $K_i(X)_{\mb{Q}}=0$.
\end{conj}
By the Bass conjecture \cite{Ba} on finite generation of $K$-groups, these groups are expected to be finite. Geisser \cite{Ge} proved that if the Tate conjecture holds and numerical and rational equivalence over finite fields agree, then Parshin's conjecture holds for all projective smooth varieties over finite fields. Kahn \cite{Ka} considered the class $\ml{B}_{tate}(k)$ of projective smooth varieties of Abelian type over $k$, for which the Tate conjecture holds. A projective smooth variety $X$ is called of Abelian type if the Chow motive of $X$ belongs to the subcategory of the category of Chow motives generated by Artin motives and Chow motives of abelian varieties. For example, products of curves and Fermat varieties are of Abelian type. For any  variety $X$ in $\ml{B}_{tate}(k)$, Kahn proved that rational and numerical equivalence agree, namely the Tate-Beilinson conjecture holds, and that Parshin's conjecture holds for $X$ using Geisser's result \cite{Ge} and Kimura's result \cite{Ki}.

Parshin's conjecture implies that for all $n\geq0$ and $i>0$, $\CH^n(X,i)_{\mb{Q}}=0$ by the isomorphism $\displaystyle K_i(X)_{\mb{Q}} \simeq \bigoplus _{n\ge0}\CH^n(X,i)_{\mb{Q}}$ (\cite{Bl}). Here $\CH^n(X,i)$ denotes Bloch's higher Chow group (\cite{Bl}) and $\CH^n(X,i)_{\mb{Q}}:=\CH^n(X,i)\otimes \mb{Q}$. By Deligne's proof of the Weil conjecture \cite{De}, we have $H^{2i-j}(\overline{X},\mb{Q}_{\ell}(i))^{G_k}=0$ for $j\neq0$. Therefore we can combine the Tate-Beilinson conjecture and the Parshin conjecture into a conjecture below. The results of Geisser and Kahn mentioned above imply that the following conjecture holds for any variety $X$ in $\ml{B}_{tate}(k)$:
\begin{conj}\label{cj}
The cycle map
\begin{align*}
\CH^i(X,j)\otimes \mb{Q}_{\ell}\longrightarrow H^{2i-j}(\overline{X},\mb{Q}_{\ell}(i))^{G_k}
\end{align*}
is bijective for all integers $i,j\geq0$.
\end{conj}
\bigskip

For a non-negative integer $r$ and a positive integer $m$ prime to $p$, we define a variety $X_m^r=X_m^r(a_0,\dots,a_{r+1}) \subset \mb{P}^{r+1}$ of dimension $r$ and of degree $m$ by the equation
$$a_0x_0^m+a_1x_1^m+\cdots +a_{r+1}x_{r+1}^m=0.$$
Here $a_i$ is a non-zero elements in $k$. In case $a_0=\cdots=a_{r+1}\neq0$, $X_m^r$ is called a Fermat variety and denoted by $V_m^r$. Shioda-Katsura \cite{SK,Sh} proved that the Tate conjecture holds for a product $V_m^{r_1}\times\cdots\times V_m^{r_s}$ of Fermat varieties of same degree under an arithmetic condition(for example, $m$ is a prime and $p \equiv 1 \mod m$).

In this paper, we are concerned with products of $X_m^r$ of different degrees, and prove the following theorem:
\begin{thm}\label{t} Let $k$ be a finite field of characteristic $p$. Let $m_1, \ldots, m_d$ be positive integers prime to $p$. Let $r_1,\dots,r_d$ be positive integers. Let $X_{m_j}^{r_j}=X_{m_j}^{r_j}({\bf a}_j)$ with coefficients ${\bf a}_j \in (k^{\times})^{r_j+2}$. Let $X$ be the product $X_{m_1}^{r_1}\times  \cdots \times X_{m_d}^{r_d}$. Then $\T(X/k,i)$ is true for all $i$ in the following cases:
\begin{itemize}
\item[\noindent$(1)$] In case that $r_j=1$ for all $j$, assume that one of the following two conditions holds:
\begin{itemize}
\item[\textup{(a)}] For each $1 \leq j\leq d$, there is at most one $j^{\prime} \neq j$ such that ${\rm gcd}(m_j,m_{j^{\prime}})> 2$.

\item[\textup{(b)}] For every even integer $j$ with $4\leq j\leq d$ and for every $1\leq n_1<n_2<\dots<n_j\leq d$, there exists an integer $a$ with $1\leq a\leq j$ such that 
\begin{itemize}
\item[\textup{(i)}] ${\rm gcd}(m_{n_a},m_{n_r})\leq 2$ for any $r \not =a$,

\item[\textup{(ii)}] the order of $p$ in the group $(\mb{Z}/m_{n_a})^{\times}$ is odd.
\end{itemize}
In this case, moreover, we obtain that $\CH^i(X)$ is the direct sum of a free abelian group of finite rank and a group of finite exponent for all $i$.
\end{itemize}
\item[\noindent$(2)$] In case that $r_j$ is odd for all $j$, assume that the following conditions hold{\rm :}

\textup{(i)} ${\rm gcd}(m_j,m_{j^{\prime}})\leq 2$ for $j \neq j^{\prime}$,

\textup{(ii)} the order of $p$ in the group $(\mb{Z}/m_j)^{\times}$ is odd for all $j$.

\item[\noindent$(3)$] In case that $r_j$ is even for all $j$, assume that the following conditions hold{\rm :}

\textup{(i)} ${\rm gcd}(m_j,m_{j^{\prime}})\leq 2$ for $j \neq j^{\prime}$,

\textup{(ii)} $\T(X_{m_j}^{r_j}/L,r_j/2)$ is true for all $j$ and for a sufficiently large finite extension $L$ of $k$. 

\item[\noindent$(4)$] In general case, assume that the following conditions hold{\rm :}

\textup{(i)} ${\rm gcd}(m_j,m_{j^{\prime}})\leq 2$ for $j \neq j^{\prime}$,

\textup{(ii)} for $j$ with odd $r_j$, the order of $p$ in the group $(\mb{Z}/m_j)^{\times}$ is odd,

\textup{(iii)} for $j$ with even $r_j$, $\T(X_{m_j}^{r_j}/L,r_j/2)$ is true for a sufficiently large finite extension $L$ of $k$. 
\end{itemize}
\end{thm} 
\begin{rmk}
For a variety $X$ in (1)--(4) of Theorem \ref{t}, the vector space $H^{2i}(\overline{X},\mb{Q}_{\ell}(i))^{G_k}$ does not vanish. Moreover, if $X$ is a variety as in (2)--(4), then the dimension of $H^{2i}(\overline{X},\mb{Q}_{\ell}(i))^{G_k}$ over $\mb{Q}_{\ell}$ can be described. For example, let $X=X_{m_1}^{r_1}\times\cdots\times X_{m_d}^{r_d}$ be as in (2). Then we have
\begin{align*}
\dim_{\mb{Q}_{\ell}}&H^{2i}(\overline{X},\mb{Q}_{\ell}(i))^{G_k}\\
&=\sharp\{(i_1,\dots,i_d)\;|\; i_1+\cdots+i_d=i \;\text{and}\; 0\le i_j\le {\rm min}\{r_j,i\}\; \text{for all $j$}\}
\end{align*}
where $\sharp $ denotes the cardinality of a finite set. The above equation follows from the proof of Theorem \ref{t} (see \S 3.1.).
\end{rmk}

Our proof of Theorem \ref{t} is based on a combinatorial property of eigenvalues of Frobenius acting on $\ell$-adic \'etale cohomology of $X_m^r$. We will use an argument similar to that Spiess used in \cite{Sp}. For (1)\,(b), we use a result of Soul\'e \cite[Th\'eor\`eme 3 i)]{So}.

From Theorem \ref{t}\,(3) and the fact that the Tate conjecture holds for $X_m^2/k$, we obtain the following:
\begin{cor}\label{cf} Let the notation be as in Theorem \ref{t}. Assume that $(m_j,m_{j^{'}})\leq 2$ for $j \neq j^{'}$. Then the Tate conjecture holds for $X_{m_1}^2\times X_{m_2}^2\times \cdots \times X_{m_d}^2/k$.
\end{cor}
In case that all coefficients are same, that is, in case that $X_{m_j}^{r_j}=V_{m_j}^{r_j}$ for all $j$, this corollary also follows from Theorem \ref{t}(1) and ``inductive structure" for Fermat varieties given by Shioda-Katsura \cite{SK} (cf. \S 3.2).

From Theorem \ref{t}, we see that the varieties as in Theorem \ref{t} belong to $B_{tate}(k)$, and hence we obtain the following (cf. Corollary \ref{ca}):
\begin{cor}\label{ct} Let $X$ be a variety as in $(1)$\textup{--}$(4)$ of Theorem \ref{t}. Then the following holds{\rm :}

$(1)$ Conjecture \ref{cj} holds for $X$,

$(2)$ The Lichtenbaum conjecture holds for $X$.

Furthermore if $X$ is a variety as in Theorem \ref{t}$\,(1)\,(b)$, then the following holds{\rm :}

\textup{(4)} $\CH^2(X)$ is finitely generated.

\textup{(5)} $K_0(X)$ is the direct sum of a free abelian group of finite rank and a group of finite exponent.
\end{cor}
Here the Lichtenbaum conjecture is the conjecture stated in \cite[\S7]{Li}.\bigskip

Corollary \ref{ct}\,(1)\,(2) immediately follow from Theorem \ref{t} and results of Geisser \cite{Ge} and Kahn \cite{Ka} mentioned above. Corollary \ref{ct}\,(4)\,(5) follow from Theorem \ref{t}\,(1)\,(b), a spectral sequence from higher Chow groups to algebraic $K$-theory \cite{FS}
$$E_2^{p,q} =\CH^{-q}(X,-p-q) \Longrightarrow K_{-p-q}(X)$$
and the following theorem:  

\begin{thm}[Colliot-Th$\acute{\text{e}}$l$\grave{\text{e}}$ne$-$Sansuc$-$Soul$\acute{\text{e}}$ \cite{CSS}]
Let $X$ be a projective smooth and geometrically connected variety over a finite field, the subgroup $\CH^2(X)_{tor}$ of torsion elements of $\CH^2(X)$ is finite.
\end{thm}\bigskip

This paper is organized as follows: in the next section (\S 2), we prove some lemmas about eigenvalues of Frobenius acting on $\ell$-adic \'etale cohomology. In \S 3, we prove Theorem \ref{t} using lemmas in \S2. In \S 3.1, we calculate the dimension of $\ell$-adic \'etale cohomology for varieties of Theorem \ref{t}. In \S 3.2, we give an alternative proof of Corollary \ref{cf} using ``inductive structure" for Fermat varieties. In the last section (\S 4), we give an application to a zeta value of products of four curves using a result of Kohmoto \cite{Ko}.

\subsection*{Notation}
For a field $k$, $k^{\times }$ denotes the multiplicative group. For an integer $m>1$, $\mb{Z}/m$ denotes the cokernel of the map $\mb{Z} \stackrel{\times m}{\rightarrow}\mb{Z}$ and $(\mb{Z}/m)^{\times }$ denotes the group of invertible elements in $\mb{Z}/m$.

Unless indicated otherwise, all cohomology groups of schemes are taken for the $\acute {\text{e}}$tale topology. Let $X$ be a scheme. For a prime number $\ell$ which is invertible on $X$ and integers $i,j\geq 0$, we define
\begin{align*}
H^i(X,\mb{Z}_{\ell}(j))&:=H^i_{\rm cont}(X,\mb{Z}_{\ell}(j)),\\
H^i(X,\mb{Q}_{\ell}(j))&:=H^i(X,\mb{Z}_{\ell}(j))\otimes _{\mb{Z}_{\ell}}\mb{Q}_{\ell}.
\end{align*} 
where $H^i_{\rm cont}(X,\mb{Z}_{\ell}(j))$ is the continuous $\ell$-adic \'etale cohomology defined by Jannsen \cite{Ja0}. We write $H^i(X,\mb{Q}_{\ell})$ for $H^i(X,\mb{Q}_{\ell}(0))$. For a scheme $X$ and a positive integer $m$ which is invertible on $X$, we write $\mu _m$ for the $\acute {\text{e}}$tale sheaf of $m$-th roots of unity on $X$. For all schemes $X$ considering in this paper, we have
\begin{align*}
H^i(X,\mb{Z}_{\ell}(j))&=\varprojlim _n H^i(X,\mu_{\ell^n}^{\otimes j}).
\end{align*} 

\section{Lemmas}
In this section, we prove some lemmas about eigenvalues of Frobenius acting on $\ell$- adic \'etale cohomology, which is used in the proof of Theorem \ref{t}.

Let $k$ be a finite field $\mb{F}_q$. We consider a variety $X_m^r$ over $k$. Let $F \in G_k$ be the geometric Frobenius. In this paper, we call an eigenvalue $\alpha$ of $F$ acting on $H^{i}(\overline{X_m^r},\mb{Q}_{\ell})$ \textit{Weil number} of weight $i$ for $X_m^r/k$. By Deligne's theorem (Weil conjecture) \cite{De}, $\alpha$ is an algebraic integer and the absolute value $|\alpha|$ of $\alpha$ is equal to $q^{i/2}$.
We have known that for any $i\neq r$, $0\leq i \leq 2r$, 
\begin{eqnarray}\label{c}
H^{i}(\overline{X_m^r},\mb{Q}_{\ell})=\left\{ \begin{array}{ll}
0 & \text{for $i$ odd} \\
\text{1-dimensional} & \text{for $i$ even}\\
\end{array} \right.
\end{eqnarray}
For this fact, we refer to \cite[Expose VII, Corollaire 7.5.]{SGA5}. In case $i$ is even, $H^{i}(\overline{X_m^r},\mb{Q}_{\ell})$ is generated by algebraic cycles, that is, the Tate conjecture is true. From \eqref{c} we see that a Weil number for $X_m^r/k$ of even weight $i \neq r$ is $q^{i/2}$, and that we have no Weil number for $X_m^r/k$ of odd weight $i \neq r$. 

Now assume that $q=p^f$ is the least power of $p$ such that $q \equiv 1 \mod m$. We then recall Weil's result on describing Weil numbers of weight $r$ for $X_m^r/k$ in terms of Jacobi sums \cite{We}. A Weil number $\alpha$ of weight $r$ for $X_m^r/k$ is given by
$$\alpha=(-1)^r\bar{\chi}(a_0)^{\gamma_0}\cdots\bar{\chi}(a_{r+1})^{\gamma_{r+1}}j(\gamma).$$
Here $\chi$ is a fixed character of order $m$ of $k^{\times}$ and $j(\gamma)$ is the Jacobi sum:

$$j(\gamma)=\sum_{\stackrel{1+v_1+\cdots+v_{r+1}=0}{ \ v_i \in k^{\times}}}\chi(v_1)^{\gamma_1}\cdots\chi(v_{r+1})^{\gamma_{r+1}},$$
where $\gamma$ is an element of the set
$$\mf{D}_{m,r}=\bigl\{ (\gamma_0,\dots,\gamma_{r+1}) \ | \ \gamma_i \in \mb{Z}/m, \ \gamma_i \not \equiv 0, \ \gamma_0+\gamma_1+\cdots+\gamma_{r+1}\equiv 0\bigr\}.$$
Conversely, for any $\gamma \in \mf{D}_{m,r}$, the number $(-1)^r\bar{\chi}(a_0)^{\gamma_0}\cdots\bar{\chi}(a_{r+1})^{\gamma_{r+1}}j(\gamma)$ is a Weil number of weight of $r$ for $X_m^r/k$.

\begin{lem}\label{l1} Let $k$ be a finite field $\mb{F}_q$. Let $\alpha$ be a Weil number of weight $r$ for $X_m^r/k$. Then there is an integer $e>0$ such that  $\alpha^e$ belongs to $\mb{Q}(\zeta_m)$, where $\zeta_m$ is a primitive $m$-th root of unity.
\end{lem}

\begin{proof}
Let $L:=\mb{F}_{q^e}$ be a finite extension of $k$ such that $q^e\equiv 1 \mod m$. Then $\alpha ^e$ is a Weil number of weight $r$ for $X_m^r/L$. By the above Weil's result,  $\alpha^e$ is a power of
$(-1)^r\bar{\chi}(a_0)^{\gamma_0}\cdots\bar{\chi}(a_{r+1})^{\gamma_{r+1}}j(\gamma)$,
and therefore belongs to $\mb{Q}(\zeta_m)$.
\end{proof}

\begin{lem}[Shioda-Katsura \text{\cite[Lemma 3.1.]{SK}}]\label{lsk} Let $f$ be the residue degree of $p$ in $\mb{Q}(\zeta_m)$. Let $\alpha$ be a Weil number of weight $r$ for $X_m^r/\mb{F}_{p^f}$. Then  the following statements are equivalent{\rm :}

$(i)$ some power of $\alpha$ is a power of $p${\rm ;}

$(ii)$ the valuations $v_{\mf{p}}(\alpha)$ of $\alpha$ are independent of a prime $\mf{p}$ in $\mb{Q}(\zeta_m)$ dividing $p$, and equal to $\frac{fr}{2}$.
\end{lem}

\begin{lem}\label{l2} Let $m_1,m_2,\dots,m_{2i}$ be positive integers prime to $p$. Let $\alpha_j$ be a Weil number of weight $1$ for $X_{m_j}^1/k$. Assume that for each $1 \leq j\leq 2i$, there is at most one $j^{\prime} \neq j$ such that ${\rm gcd}(m_j,m_{j^{\prime}})> 2$. If $\alpha_1\alpha_2\cdots\alpha_{2i}=q^i$, then$-$after renumbering the $\alpha_j$ if necessary$-$there exists an integer $N$ such that 
\begin{align*}
(\alpha_1\alpha_2)^N=\cdots=(\alpha_{2i-1}\alpha_{2i})^N=q^N.
\end{align*}
\end{lem}

\begin{proof}
From the assumption, after renumbering the $m_j$ if necessary, we have pairs $(m_1,m_2),\dots,(m_{2i-1},m_{2i})$ such that 
$${\rm gcd}({\rm lcm}(m_{2j-1},m_{2j}),{\rm lcm}(m_{2j^{\prime}-1},m_{2j^{\prime}}))\leq2 \ \text{ for all } \ j \neq j^{\prime}.$$
Put $L_j:={\rm lcm}(m_{2j-1},m_{2j})$ for $1\leq j\leq i$. From the equation $\alpha_1\alpha_2\cdots\alpha_{2i}=q^i$ and Lemma \ref{l1}, there is a positive integer $M$ such that
\begin{align*}
(\alpha_1\alpha_2)^M=q^{Mi}(\alpha_3\cdots\alpha_{2i})^{-M} \in \mb{Q}(\zeta_{L_1})\cap \mb{Q}(\{\zeta_{L_j},j\neq 1\}).
\end{align*}
Now we have
$$\mb{Q}(\zeta_{L_1})\cap \mb{Q}(\{\zeta_{L_j},j\neq 1\}) \subset \mb{Q}(\zeta_{L_1})\cap \mb{Q}(\zeta_L)=\mb{Q}.$$
Here $L={\rm lcm}(L_j,j\neq 1)$. Hence $(\alpha_1\alpha_2)^M$ belongs to $\mb{Q}$ and therefore $(\alpha_1\alpha_2)^{2M}=|(\alpha_1\alpha_2)^M|^2=q^{2M}$, because $\alpha$ is a Weil number of weight $1$. The same argument shows that $(\alpha_{2i-1}\alpha_{2i})^{2M}=q^{2M}$ for $2\leq i\leq d$.
\end{proof}

\begin{lem}\label{l3}
Let $m_1, \dots, m_d$ be positive integers prime to $p$. Let $r_1, \dots, r_d$ be positive odd integers. For $j=1,\dots,d$, let $\alpha_j$ be a Weil number of weight $i_j$ for $X_{m_j}^{r_j}/k$. Assume that there is an integer $j_0$ such that

$(1)$ $(m_{j_0},m_j)\leq 2$ for $j \neq j_0$,

$(2)$ $i_{j_0}=r_{i_{j_0}}$,

$(3)$ the order of $p$ in the group $(\mb{Z}/m_{j_0})^{\times}$ is odd.\\
Then the product $\alpha_1\alpha_2\cdots\alpha_d$ does not belong to $\mb{Q}$.
\end{lem}

\begin{proof}
We may assume that $j_0 =1$. From Lemma \ref{l1}, there is a positive integer $N$ such that $\alpha_j^N \in \mb{Q}(\zeta_{m_j})$ for all $j$. If $\alpha_1\alpha_2\cdots\alpha_d=:b \in \mb{Q}$, then we have
\begin{align*}
\alpha_1^N = b^N(\alpha_2\cdots\alpha_d)^{-N} \in \mb{Q}(\zeta_{m_1})\cap\mb{Q}(\zeta_{m_j}, j\neq 1).
\end{align*}
From the assumption (1), $\mb{Q}(\zeta_{m_1})\cap\mb{Q}(\zeta_{m_j}, j\neq 1)=\mb{Q}$. By the assumption (2), we have $|\alpha_1^N|=q^{Nr_1/2}$. Therefore we see that some power of $\alpha_1$ is a power of $p$.

On the other hand, by linear algebra, we see that some power of $\alpha_1$ is a power of a Weil number $\alpha$ of weight $r_1$ for $X_{m_1}^{r_1}/\mb{F}_{p^f}$. Here $f$ is the order of $p$ in $(\mb{Z}/m_1)^{\times}$ which is odd by assumption (3). Since $\frac{fr_1}{2}$ is not an integer, the valuation $v_{\mf{p}}(\alpha)$ is not equal to $\frac{fr_1}{2}$ for all primes $\mf{p}$ in $\mb{Q}(\zeta_{m_1})$ dividing $p$. By Lemma \ref{lsk}, any power of $\alpha$ is not a power of $p$. Therefore any power of $\alpha_1$ is not a power of $p$, which is a contradiction. 
\end{proof}

\begin{lem}\label{l4} Let $m_1, \dots, m_d$ be positive integers prime to $p$. Let $r_1, \dots, r_d$ be positive integers. Assume that ${\rm gcd}(m_j,m_{j^{\prime}})\leq 2$ for $j \neq j^{\prime}$. For $j=1,\dots,d$, let $\alpha_j$ be a Weil number of weight $2r_j$ for $X_{m_j}^{2r_j}/k$. Put $r:=r_1+\cdots+r_d$. If $\alpha_1\alpha_2\cdots\alpha_d=q^r$, then there exists an integer $N$ such that 
\begin{align*}
\alpha_1^N=q^{Nr_1}, \cdots, \alpha_d^N=q^{Nr_d}.
\end{align*}
\end{lem}

\begin{proof}
From the equation $\alpha_1\alpha_2\cdots\alpha_d=q^r$ and Lemma \ref{l1}, there is a positive integer $M$ such that 
$$\alpha_1^M=q^{Mr}(\alpha_2\cdots\alpha_d)^{-M} \in \mb{Q}(\zeta_{m_1})\cap \mb{Q}(\{\zeta_{m_j},j\neq 1\}).$$
By the assumption, we have
$$\mb{Q}(\zeta_{m_1})\cap \mb{Q}(\{\zeta_{m_j},j\neq 1\})=\mb{Q}.$$
Hence $\alpha_1^M \in \mb{Q}$ and therefore $\alpha_1^{2M}=|\alpha_1^M|^2=q^{2Mr_1}$, because $\alpha_1$ is a Weil number of weight $r_1$. The same argument shows that $\alpha_j^{2M}=q^{2Mr_j}$ for $2\le j\le d$. 
\end{proof}

\section{Proof of the theorem}
Before beginning to prove the theorem, we define a class $\ml{A}(k)$ constructed from projective smooth varieties as in Theorem \ref{t} and prove that Conjecture \ref{cj} in the introduction holds for varieties which belong to $\ml{A}(k)$. The argument is based on that of Soul\'e \cite{So}.

Let $Y$ be a projective smooth variety over a finite field $k$ which is not necessarily geometrically integral over $k$. Put $\kappa :=\Gamma (Y,\ml{O}_Y)$. Then the scalar extension $Y\otimes _k\kappa $ is a disjoint union of copies of $Y$, and we have
\begin{align*}
\CH^i(Y)\otimes \mb{Q}_{\ell}&\simeq \bigl(\CH^i(Y\otimes _k\kappa )\otimes \mb{Q}_{\ell}\bigr)^{{\rm Gal}(\kappa /k)}\\
&\simeq \bigl(\CH^i(Y)\otimes \mb{Q}_{\ell}[{\rm Gal}(\kappa/k)]\bigr)^{{\rm Gal}(\kappa /k)}.
\end{align*}
Hence the bijectivity of the cycle class map \eqref{map} for $Y/k$ is equivalent to that for $Y/\kappa $. For this reason we also consider the case that $Y$ is not geometrically integral.

\begin{df}\label{ak} For any finite field $k$, we define $\ml{A}(k)$ to be the smallest class of projective smooth varieties over $k$ satisfying the following properties:

(1) Varieties which satisfy the assumption of Theorem \ref{t} belong to $\ml{A}(k)$.

(2) If $X$ and $Y$ belong to $\ml{A}(k)$, then the disjoint union $X\amalg Y$ of $X$ and $Y$ belongs to $\ml{A}(k)$.

(3) If $X$ belongs to $\ml{A}(k)$, $Y$ is a projective smooth variety such that $\dim X = \dim Y$, $Y$ is a direct summand of $X$ as Chow motives with $\mb{Z}[\frac{1}{n}]$ coefficients for some $n\geq1$, then $Y$ belongs to $\ml{A}(k)$.

(4) Let $X$ be a projective smooth variety over $k$ and $k^{\prime }$ be a finite extension of $k$. If $X\otimes k^{\prime }$ belongs to $\ml{A}(k^{\prime})$, then $X$ belongs to $\ml{A}(k)$.

(5) If $X$ belongs to $\ml{A}(k)$ and $E$ is a vector bundle on $X$, then the projective bundle $P(E)$ on $X$ associated to $E$ belongs to $\ml{A}(k)$.

(6) Let $X$ be a projective smooth variety over $k$ and $Y$ be a closed smooth subvariety of $X$. Let $W$ be the blowing up of $X$ along $Y$. Then $W$ belongs to $\ml{A}(k)$ if and only if $X$ and $Y$ belong to $\ml{A}(k)$.
\end{df}
\begin{cor}\label{ca}
Let $X$ be a variety which belong to $\ml{A}(k)$. Then we have the following{\rm :}

$(1)$ Conjecture \ref{cj} holds for $X$,

$(2)$ The Lichtenbaum conjecture holds for $X$.
\end{cor}

\begin{proof}
By results of Kahn \cite{Ka} and Geisser \cite{Ge}, it suffices to show that the Tate-Beilinson conjecture holds for all $i$ and for $X \in \ml{A}(k)$.

Let $\ml{A}^{\prime }(k)$ be a class of all projective smooth varieties over $k$ for which the Tate-Beilinson conjecture holds. Then by the smallness of $\ml{A}(k)$, it suffices to show that $\ml{A}^{\prime }(k)$ satisfies the above properties (1)--(6). For (1), the assertion follows from Theorem \ref{t}. We only show that $\ml{A}^{\prime }(k)$ satisfies property (6). The other properties (2)--(5) can be checked by similar arguments. Let the notation be as in property (6). Since we have the following decomposition of the Chow motive $\widetilde{W}$ of $W$ (cf.\ \cite{So})
\begin{align*}
\widetilde{W}\simeq \widetilde{X}\oplus \bigl(\bigoplus_{j=1}^{c-1}\widetilde{Y}\otimes \mb{L}^j\bigr),
\end{align*}
we have the following isomorphisms (cf.\ \cite[Th$\acute{\text{e}}$or$\grave {\text{e}}$me 4.]{So}):
\begin{align*}
\CH^i(W)&\simeq \CH^i(X)\oplus \big( \bigoplus_{j=1}^{c-1}\CH^{i-j}(Y)\big),\\
H^{2i}(\overline{W},\mb{Q}_{\ell}(i))&\simeq H^{2i}(\overline{X},\mb{Q}_{\ell}(i))\oplus \big( \bigoplus_{j=1}^{c-1}H^{2(i-j)}(\overline{Y},\mb{Q}_{\ell}(i-j))\big).
\end{align*}
Hence $W$ belongs to $\ml{A}^{\prime }(k)$ if and only if $X$ and $Y$ belong to $\ml{A}^{\prime }(k)$.
\end{proof}

We prove Theorem \ref{t} using Lemmas about Weil numbers for $X_m^r/k$ in \S 2. We can formulate the assertion of the Tate conjecture for Chow motives and simplify the argument. For the definition and basic properties  of Chow motives, we refer to \cite{So}. Spiess \cite{Sp} used Chow motives in the proof of the Tate conjecture for products of elliptic curves. We use similar arguments of Spiess. We first give the following Lemma:
\begin{lem}\label{l5}
Let $L/k$ be a finite extension and $X$ be a projective smooth variety over $k$. Then $\T(X\times_kL/L,i)$ implies $\T(X/k,i)$.
\end{lem}
For the proof of this lemma, see \cite[Lemma 1]{Sp}.  
\renewcommand{\proofname}{\rm \bf Proof of Theorem \ref{t}\,(1)\,(a)}
\begin{proof}
Using the motivic decomposition of the Chow motive $X_m^1=1\oplus X_m^{+} \oplus \mb{L}$ (see \cite{So}), our task is reduced to show that $\T(X_{m_1}^{+}\otimes X_{m_2}^{+}\otimes \cdots\otimes X_{m_s}^{+}/k,i)$ holds for all $s$ and $i$. Since Soul\'e proved that $H^{2i}(\overline{X_{m_1}^{+}\otimes X_{m_2}^{+}\otimes \cdots\otimes X_{m_s}^{+}},\mb{Q}_{\ell}(i))=0$ for $s\neq 2i$, we may assume $s=2i$. We have
\begin{align*}
H^{2i}(\overline{X_{m_1}^{+}\otimes X_{m_2}^{+}\otimes \cdots\otimes X_{m_s}^{+}},\mb{Q}_{\ell}(i))^{G_k}\simeq \bigl(H_{1}\otimes H_{2}\otimes \cdots \otimes H_{2i}\bigr)^{F=q^i}.
\end{align*}
Here we set $H_{s}=H^1(\overline{X_{m_s}^1},\mb{Q}_{\ell})$ and $F \in G_k$ is the geometric Frobenius.

A basis of the vector space $\bigl(H_{1}\otimes H_{2}\otimes \cdots \otimes H_{2i}\bigr)^{F=q^i}$ corresponds to $2i$-tuples $(\alpha_1,\alpha_2,\dots,\alpha_{2i})$ such that $\alpha_1\alpha_2\cdots\alpha_{2i}=q^i$ (where $\alpha _s$ is a Weil number of weight $1$ for $X_{m_s}^1/k$). 

From the assumption (a) in Theorem \ref{t} and Lemma \ref{l2}, there is a positive integer $N$ such that for every such tuples $(\alpha_1,\alpha_2,\dots,\alpha_{2i})$, after renumbering the $\alpha_j$ if necessary, we have
\begin{align*}
(\alpha_1\alpha_2)^N=\cdots=(\alpha_{2i-1}\alpha_{2i})^N=q^N.
\end{align*}
From this we obtain that the map
\begin{align*}
\bigoplus_{\sigma}\bigotimes _{j=1}^iH^{2}(\overline{X_{m_{\sigma(2j-1)}}^{+}\otimes X_{m_{\sigma(2j)}}^{+}}&,\mb{Q}_{\ell}(1))^{G_L} \\
&\longrightarrow H^{2i}(\overline{X_{m_1}^{+}\otimes X_{m_2}^{+}\otimes \cdots\otimes X_{m_{2i}}^{+}},\mb{Q}_{\ell}(i))^{G_L}
\end{align*}
is surjective, where $\sigma$ runs through the set of permutations of $\{1,2,\dots,2i\}$ and $L$ is a finite extension of $k$ with $[L:k]=N$. Now the assertion follows from Lemma \ref{l5}, $\T(X_m^1\times X_{m^{\prime}}^1\times_kL/L,1)$(it is true) and the commutative diagram
\begin{align*}{\scriptsize 
\xymatrix{\displaystyle 
\bigoplus_{\sigma}\bigotimes _{j=1}^i\CH^1(X_{m_{\sigma(2j-1)}}^{+}\otimes X_{m_{\sigma(2j)}}^{+}\otimes L)\otimes \mb{Q}_{\ell} \ar[r] \ar[d] &\displaystyle\bigoplus_{\sigma}\bigotimes _{j=1}^iH^{2}(\overline{X_{m_{\sigma(2j-1)}}^{+}\otimes X_{m_{\sigma(2j)}}^{+}},\mb{Q}_{\ell}(1))^{G_L} \ar@{>>}[d]\\
\CH^i(X_{m_1}^{+}\otimes X_{m_2}^{+}\otimes \cdots\otimes X_{m_{2i}}^{+}\otimes L)\otimes\mb{Q}_{\ell}\ar[r]&H^{2i}(\overline{X_{m_1}^{+}\otimes X_{m_2}^{+}\otimes \cdots\otimes X_{m_{2i}}^{+}},\mb{Q}_{\ell}(i))^{G_L}.
}}
\end{align*}
\end{proof}
\renewcommand{\proofname}{\rm \bf Proof of Theorem \ref{t}\,(1)\,(b)}
\begin{proof}
We first recall a result of Soul\'e (\cite[Theorem 3]{So}).

Let $C_1,\dots,C_d$ be projective, smooth and geometrically irreducible curves over $k=\mb{F}_{q}$, and $X=C_1\times \cdots\times C_d$ be the product. We consider the following condition on $X/k$ and integer $0\leq i\leq d$:\bigskip\\
{\bf Condition $\boldsymbol{(*)_i}$} \ Let $j$ be an even integer with $4\leq j\leq  \inf(2i,2d-2i)$. For any $1\leq n_1<n_2<\dots<n_j\leq d$, if $\alpha _{n_1},\dots,\alpha _{n_j}$ are eigenvalues of $F$ acting on  $H^1(\overline{C}_{n_1},\mb{Q}_{\ell}),\dots,H^1(\overline{C}_{n_j},\mb{Q}_{\ell})$, then the product $\alpha_{n_1} \alpha_{n_2}\cdots \alpha_{n_j} $ is not equal to $q^{j/2}$.

\begin{thm}[Soul$\acute{\text{e}}$ \cite{So}]\label{ts}
Let $i$ be a positive integer with $0\leq i\leq d$. If $X$ satisfies the condition $(*)_i$ over $k$, then $\CH^i(X)$ is the direct sum of a free abelian group of finite rank and a group of finite exponent. Moreover the cycle map $\rho_X^i$ is bijective.
\end{thm}

\begin{rmk}
The condition $(*)_i$ holds for $i=0,1,d-1,d$. Therefore if $d\leq 3$, the Tate-Beilinson conjecture holds by Theorem \ref{ts}.
\end{rmk}

From Theorem \ref{ts}, our task is reduced to show that $X$ satisfies the condition $(*)_i$ for all $i$. In our case, from the assumption (b) in Theorem \ref{t} and Lemma \ref{l3}, we see that the product $\alpha_{n_1} \alpha_{n_2}\cdots \alpha_{n_j}$ does not belong to $\mb{Q}$. Hence $X$ satisfies the condition $(*)$ for all $i$. 
\end{proof}

\renewcommand{\proofname}{\rm \bf Proof of Theorem \ref{t}\,(2)}
\begin{proof}
From K\"unneth formula, we have
\begin{align*}
 H^{2i}(&\overline{X_{m_1}^{r_1}\times X_{m_2}^{r_2}\times \cdots \times X_{m_d}^{r_d}},\mb{Q}_{\ell}(i))^{G_k}\simeq \bigoplus_{i_1+\cdots+i_d=2i}  W(i_1,\dots,i_d),
 \end{align*}
where $$W(i_1,\dots,i_d):=\bigl(H^{i_1}(\overline{X_{m_1}^{r_1}},\mb{Q}_{\ell})\otimes \cdots \otimes H^{i_d}(\overline{X_{m_d}^{r_d}},\mb{Q}_{\ell})\bigr)^{F=q^i}.$$ From \eqref{c}, Lemma \ref{l3} and the assumption of the theorem, we obtain that if $i_j$ is odd for some $j$, then $W(i_1,\dots,i_d)=0$.
 
In the remaining case where $i_j$ is even for all $j$, from \eqref{c} we have
\begin{align*}
\bigotimes _{j=1}^dH^{i_j}(\overline{X_{m_j}^{r_j}},\mb{Q}_{\ell}(i_j/2))^{G_k} \simeq W(i_1,\dots,i_d).
\end{align*}
We know that if $r$ is odd, then the Tate conjecture $\T(X_m^{r}/k, i)$ is true for all $i$. Therefore the assertion follows from a similar argument in the proof of Theorem \ref{t}\,(1)\,(a).  
\end{proof}

\renewcommand{\proofname}{\rm \bf Proof of Theorem \ref{t}\,(3)}
\begin{proof}
Let $i$ be an integer with $0\leq i\leq r$, where $r=r_1+\cdots+r_d$. From K\"unneth formula and isomorphisms \eqref{c}, we have
\begin{align*}
 H^{2i}(&\overline{X_{m_1}^{r_1}\times X_{m_2}^{r_2}\times \cdots \times X_{m_d}^{r_d}},\mb{Q}_{\ell}(i))^{G_k}\simeq \bigoplus_{i_1+\cdots+i_{d}=i}  V(i_1,\dots,i_{d})^{G_k},
 \end{align*}
where $$V(i_1,\dots,i_{d}):=\bigl(H^{2i_1}(\overline{X_{m_1}^{r_1}},\mb{Q}_{\ell}(i_1))\otimes \cdots \otimes H^{2i_{d}}(\overline{X_{m_{d}}^{r_{d}}},\mb{Q}_{\ell}(i_{d}))\bigr).$$
Similarly to the above proof of (1)\,(a), a basis of $V(i_1,\dots,i_{d})^{G_k}$ corresponds to $d$-tuples $(\alpha_1,\alpha_2,\dots,\alpha_d)$ such that $\alpha_1\alpha_2\cdots\alpha_d=q^i$ (where $\alpha _j$ is a Weil number of weight $2i_j$ for $X_{m_j}^{r_j}/k$).

From assumption (i) of the theorem, Lemma \ref{l4} and isomorphism \eqref{c}, there is a positive integer $N$ such that for every such $d$-tuples $(\alpha_1,\alpha_2,\dots,\alpha_d)$, we have  
\begin{align*}
\alpha_1^N=q^{Ni_1}, \cdots, \alpha_d^N=q^{Ni_d}.
\end{align*}
From this we obtain that the map
\begin{align*}
\bigotimes _{j=1}^{d}H^{2i_j}(\overline{X_{m_j}^{r_j}},\mb{Q}_{\ell}(i_j))^{G_L} \longrightarrow V(i_1,\dots,i_{d})^{G_L}
\end{align*}
is surjective, where $L$ is a finite extension of $k$ with $[L:k]=N$. The assertion now follows from assumption (ii) of the theorem and the commutative diagram
\begin{align*}{\footnotesize 
\xymatrix{\displaystyle 
\bigoplus_{i_1+\cdots+i_{d}=i \ \ }\bigotimes _{j=1}^{d}\CH^{i_j}(X_{m_j}^{r_j}\times L)\otimes \mb{Q}_{\ell} \ar[r] \ar[d] &\displaystyle\bigoplus_{i_1+\cdots+i_{d}=i \ \ }\bigotimes _{j=1}^dH^{2i_j}(\overline{X_{m_j}^{r_j}},\mb{Q}_{\ell}(i_j))^{G_L} \ar@{>>}[d]\\
\CH^i(X_{m_1}^{r_1}\times X_{m_2}^{r_2}\times \cdots\times X_{m_{d}}^{r_{d}}\times L)\otimes\mb{Q}_{\ell}\ar[r]&H^{2i}(\overline{X_{m_1}^{r_1}\times X_{m_2}^{r_2}\times \cdots \times X_{m_{d}}^{r_{d^{\prime}}}},\mb{Q}_{\ell}(i))^{G_L}.
}}
\end{align*}
\end{proof}

\renewcommand{\proofname}{\rm \bf Proof of Theorem \ref{t}\,(4)}
\begin{proof}
Let $Y$ be the product of all factors of $X$ of odd dimension and let $Z$ be the product of all factors of $X$ of even dimension. From K\"unneth formula and isomorphisms \eqref{c}, we have
\begin{align*}
H^{2i}(\overline{X},\mb{Q}_{\ell}) \simeq \bigoplus_{i_1+i_2=i} H^{2i_1}(\overline{Y},\mb{Q}_{\ell})\otimes H^{2i_2}(\overline{Z},\mb{Q}_{\ell}).
\end{align*}

A basis of $\bigl(H^{2i_1}(\overline{Y},\mb{Q}_{\ell}(i_1))\otimes H^{2i_2}(\overline{Z},\mb{Q}_{\ell}(i_2))\bigr)^{G_k}$ corresponds to pairs $(\alpha,\beta)$ such that $\alpha\beta=q^i$ where $\alpha$ and $\beta$ are an eigenvalue of geometric Frobenius acting on $H^{2i_1}(\overline{Y},\mb{Q}_{\ell})$ and $H^{2i_2}(\overline{Z},\mb{Q}_{\ell})$ respectively. 

Let $m$ (resp. $n$) be the least common multiple of $\{m_j\;|\; \text{$r_j$ is odd (resp. even)}\}$. From assumption (i) of the theorem, ${\rm gcd}(m,n)\leq2$. From K\"unneth formula and Lemma \ref{l1}, there is a positive integer $N$ such that for every such pairs $(\alpha,\beta)$, we have
\begin{align*}
\alpha^N=q^i\beta^{-N} \in \mb{Q}(\zeta_m)\cap \mb{Q}(\zeta_n)=\mb{Q}.
\end{align*}
From this we obtain that the map
\begin{align*}
\bigoplus_{i_1+i_2=i} H^{2i_1}(\overline{Y},\mb{Q}_{\ell}(i_1))^{G_E}\otimes H^{2i_2}(\overline{Z},\mb{Q}_{\ell}(i_2))^{G_E}\longrightarrow H^{2i}(\overline{X},\mb{Q}_{\ell}(i))^{G_E}
\end{align*}
is surjective, where $E/k$ is a finite extension of degree $N$. Hence the assertion follows from Theorem \ref{t}\,(2)\,(3) and a commutative diagram
\begin{align*}{\footnotesize
\xymatrix{\displaystyle 
\bigoplus_{i_1+i_2=i}\CH^{i_1}(Y\times E)\otimes \CH^{i_2}(Z\times E)\otimes \mb{Q}_{\ell} \ar[r] \ar[d] &\displaystyle\bigoplus_{i_1+i_2=i \ \ }H^{2i_1}(\overline{Y},i_1)^{G_E}\otimes  H^{2i_2}(\overline{Z},i_2)^{G_E}\ar@{>>}[d]\\
\CH^i(X\times E)\otimes\mb{Q}_{\ell}\ar[r]&H^{2i}(\overline{X},\mb{Q}_{\ell}(i))^{G_E}.
}}
\end{align*}
Here $H^{2i_1}(\overline{Y},i_1)$ and $H^{2i_2}(\overline{Z},i_2)$ denote $H^{2i_1}(\overline{Y},\mb{Q}_{\ell}(i_1))$ and $H^{2i_2}(\overline{Z},\mb{Q}_{\ell}(i_2))$ respectively.
\end{proof}
\renewcommand{\proofname}{\textit{Proof}}
%%%%%%%%%%%%%%%%%%%%%%%%%%%%%%%%%%
%%%%%%%%%%%%%%%%%%%%%%%%%%%%%%%%%%
%%%%%%%%%%%%%%%%%%%%%%%%%%%%%%%%%%

\subsection{Dimension of $\ell$-adic cohomology}
For a variety $X$ as in Theorem \ref{t}, the $\mb{Q}_{\ell}$-vector space $H^{2i}(\overline{X},\mb{Q}_{\ell}(i))^{G_k}$ does not vanish. In case where $X$ is a variety as in (2) of Theorem \ref{t}, $H^{2i}(\overline{X},\mb{Q}_{\ell}(i))^{G_k}$ comes from tensor products of Lefschetz motives. In other cases, $H^{2i}(\overline{X},\mb{Q}_{\ell}(i))^{G_k}$ may contain a subspace which does not comes from tensor products of Lefschetz motives. For example, $H^{2}(\overline{X_m^1\times X_m^1},\mb{Q}_{\ell}(1))^{G_k}$ contains the subspace which is isomorphic to ${\rm Hom}_{\mb{Q}_{\ell}}(V_{\ell}(J_m),V_{\ell}(J_m))^{G_k}$. Here $J_m$ is the Jacobian variety of $X_m^1$ over $k$ and $V_{\ell}(J_m)$ is defined as follows: let $J_m[n]$ denote the group of elements $x \in J_m(\overline{k})$ such that $nx=0$. We define $T_{\ell}(J_m)$ to be the projective limit of the groups $J_m[\ell^n]$ with respect to the maps $J_m[\ell^{n+1}]\stackrel{\times \ell}{\longrightarrow} J_m[\ell^n]$. Then we define $V_{\ell}(J_m)=\mb{Q}_{\ell}\otimes_{\mb{Z}_{\ell}}T_{\ell}(J_m)$. It is well known that $T_{\ell}(J_m)$ is a free module over $\mb{Z}_{\ell}$ of rank $(m-1)(m-2)$. Therefore $V_{\ell}(J_m)$ is a $\mb{Q}_{\ell}$-vector space of dimension $(m-1)(m-2)$.\bigskip

The dimension of $H^{2i}(\overline{X},\mb{Q}_{\ell}(i))^{G_k}$ over $\mb{Q}_{\ell}$ can be computed for a variety $X$ as in (2)--(4) of Theorem \ref{t}.

First let $X=X_{m_1}^{r_1}\times\cdots\times X_{m_d}^{r_d}$ be a variety as in (2). By the proof of Theorem \ref{t}\,(2), we have
\begin{align*}
H^{2i}(&\overline{X},\mb{Q}_{\ell}(i))^{G_k}\\
&=\bigoplus_{(i_1,\dots,i_d)\in I(i)}H^{2i_1}(\overline{X_{m_1}^{r_1}},\mb{Q}_{\ell}(i_1))^{G_k}\otimes \cdots \otimes H^{2i_d}(\overline{X_{m_d}^{r_d}},\mb{Q}_{\ell}(i_d))^{G_k}
\end{align*}
where $$I(i)=\Big\{(i_1,\dots,i_d)\;\Big|\; i_1+\cdots+i_d=i \;\text{and}\; 0\le i_j\le {\rm min}\{r_j,i\}\; \text{for all $j$}\Big\}.$$ The direct summand of right hand side is isomorphic to $\mb{Q}_{\ell}$. Hence we have
\begin{align}\label{d1}
\dim_{\mb{Q}_{\ell}}H^{2i}(\overline{X},\mb{Q}_{\ell}(i))^{G_k}=\sharp I(i)
\end{align}
where $\sharp $ denotes the cardinality of a finite set. 

Next let $X=X_{m_1}^{r_1}\times\cdots\times X_{m_d}^{r_d}$ be a variety as in (3). By the proof of Theorem \ref{t}\,(3), we have
\begin{align*}
H^{2i}(\overline{X},\mb{Q}_{\ell}(i)&)^{G_k}\\
=\bigoplus_{(i_1,\dots,i_{d})\in I^{\prime}(i)}&H^{2i_1}(\overline{X_{m_1}^{r_1}},\mb{Q}_{\ell}(i_1))^{G_k}\otimes \cdots \otimes H^{2i_{d}}(\overline{X_{m_{d}}^{r_{d}}},\mb{Q}_{\ell}(i_{d}))^{G_k}\\
&\oplus \bigoplus _{(i_1,\dots,i_{d})\in I^{''}(i)} H^{2i_1}(\overline{X_{m_1}^{r_1}},\mb{Q}_{\ell}(i_1))^{G_k}\otimes \cdots \otimes H^{2i_{d}}(\overline{X_{m_{d}}^{r_{d}}},\mb{Q}_{\ell}(i_{d}))^{G_k}
\end{align*}
where $$\displaystyle I^{\prime}(i)=\Big\{(i_1,\dots,i_{d})\;\Big|\; \sum_{j=1}^{d}i_j=i,\; 0\le i_j\le {\rm min}\{r_j,i\}\;\text{and}\; 2i_j\neq r_j\; \text{for all $j$}\Big\}$$ and 
\begin{align*}
I^{''}(i)=\Bigg\{(i_1,\dots,i_{d})\;\Bigg| 
\begin{array}{l}
\displaystyle \sum_{j=1}^{d}i_j=i,\; 0\le i_j\le {\rm min}\{r_j,i\}\;\text{for all $j$,}\\
2i_j= r_j\; \text{for some $j$}
\end{array}\Bigg\}
\end{align*}
The direct summand of the first part of the right hand side is isomorphic to $\mb{Q}_{\ell}$. Hence we have
\begin{align}\label{d2}
\dim_{\mb{Q}_{\ell}}H^{2i}(\overline{X},\mb{Q}_{\ell}(i))^{G_k}=\sharp I^{\prime}(i)+\sum_{(i_1,\dots,i_{d^{\prime}})\in I^{''}(i)}\,\prod_{2i_j=r_j} \dim_{\mb{Q}_{\ell}}H^{r_j}(\overline{X_{m_j}^{r_j}},\mb{Q}_{\ell}(i_j))^{G_k}.
\end{align}
The description of $\dim_{\mb{Q}_{\ell}}H^{r_j}(\overline{X_{m_j}^{r_j}},\mb{Q}_{\ell}(r_j/2))^{G_k}$ is known (cf. \cite[p.\ 125]{Sh3}). We recall it here. We use the notation in \S2.1. Let $m$ be a positive integer prime to $p$ and let $r$ be a positive even integer. We denote the cardinality of $k$ by $q$. Then we have 
\begin{align*}
\dim_{\mb{Q}_{\ell}}H^{r}(\overline{X_{m}^{r}},\mb{Q}_{\ell}(r/2))^{G_k}=1+\sharp \mf{B}_{m,r,q},
\end{align*}
where 
$$\mf{B}_{m,r,q}=\big\{\gamma \in \mf{D}_{m,r}\;|\;j(\gamma)=q^{r/2}\big\}.$$
To give another description of $\mf{B}_{m,r,q}$, we introduce some notation. For $\gamma=(\gamma_0,\dots,\gamma_{r+1}) \in \mf{D}_{m,r}$, we define 
$$\parallel \gamma \parallel=\sum_{i=1}^{r+1}\big\langle \frac{\gamma_i}{m}\big\rangle-1,$$
where $\langle x\rangle$ is the fractional part of $x\in \mb{Q}/\mb{Z}$. We write $H$ for the subgroup of $(\mb{Z}/m)^{\times}$ generated by $p$, and $f$ for the order of $H$. Then for a sufficiently large $q$,  the set $\mf{B}_{m,r,q}$ is equal to the following set
$$\Big\{\gamma \in \mf{D}_{m,r}\;\Big|\;\sum_{h \in H}\parallel ht\gamma \parallel=rf/2, \;{}^{\forall}t \in (\mb{Z}/m)^{\times}\Big\}.$$

Lastly let $X$ be a variety as in (4). Let $Y$ and $Z$ be the varieties as in the proof of Theorem \ref{t} (4) (i.e. $X=Y \times Z$). By the proof, we have
$$H^{2i}(\overline{X},\mb{Q}_{\ell}(i))^{G_k}\simeq \bigoplus_{i_1+i_2=i} H^{2i_1}(\overline{Y},\mb{Q}_{\ell}(i_1))^{G_k}\otimes H^{2i_2}(\overline{Z},\mb{Q}_{\ell}(i_2))^{G_k}.$$
Therefore the dimension of $H^{2i}(\overline{X},\mb{Q}_{\ell}(i))^{G_k}$ can be computed using the above descriptions \eqref{d1} \eqref{d2}.

%%%%%%%%%%%%%%%%%%%%%%%%%%%%%%
%%%%%%%%%%%%%%%%%%%%%%%%%%%%%%
%%%%%%%%%%%%%%%%%%%%%%%%%%%%%%

\subsection{Alternative proof of Corollary \ref{cf}}
Let $k$ be a finite field of characteristic $p$ and $m$ be a positive integer prime to $p$. For an integer $r\geq 0$, let $V_m^r \subset \mb{P}_{k}^{r+1}$ be the Fermat variety of dimension $r$ and of degree $m$ defined by the following equation
\begin{align*}
x_0^m+x_1^m+x_2^m+\cdots +x_{r+1}^m=0.
\end{align*}
Shioda and Katsura \cite{SK} gave the ``inductive structure" of Fermat varieties of a common degree and of various dimensions, and proved that the Tate conjecture holds for Fermat surfaces using this structure. We here give an alternative proof of Corollary \ref{cf} using the inductive structure of Fermat varieties and Theorem \ref{t}\,(1).

\begin{thm}[Shioda-Katsura \cite{SK}, Shioda \cite{Sh2}]\label{isf}
Let $m$ be a positive integer prime to $p$ and $k$ be a finite field with $\zeta_{2m} \in k$. Let $r$ and $s$ be positive integers. Then there is a commutative diagram{\rm :}
\begin{align*}
\xymatrix{
& Z_m^{r,s} \ar[d]^{\beta } \ar[r]^{\pi } &Z_m^{r,s}/\mu_m \ar[d]^{\overline{\psi} }\\
V_m^{r-1}\times V_m^{s-1} \ar@{^{(}->}[r]^{ \ j}&V_m^r\times V_m^s \ar@{.>}[r]^{\varphi } &V_m^{r+s} & V_m^{r-1}\amalg V_m^{s-1}\ar@{_{(}->}[l]_{i \ \ }.
}
\end{align*}
Here $\mu_m$ is the group of $m$-th roots of unity and the above maps are defined as follows{\rm :}

$\varphi :$ the rational map of degree $m$ defined by 

\ \ \ \ $\varphi ((x_0:\cdots:x_{r+1}),(y_0:\cdots:y_{s+1}))$

 \ \ \ \ \ \ $=(x_0y_{s+1}:\cdots:x_{r+1}y_{s+1}:\zeta_{2m}x_{r+1}y_0:\cdots:\zeta_{2m}x_{r+1}y_{s+1})$,
 
$j : (((x_0:\cdots:x_r),(y_0:\cdots:y_s)))\mapsto ((x_0:\cdots:x_r:0),(y_0:\cdots:y_s:0))$,

$i=1_1\amalg i_2 :$ 

\ \ \ \ $\begin{cases}
i_1 : (x_0:\cdots:x_r)\longmapsto (x_0\cdots:x_r:0:\cdots:0)\\
i_2 : (y_0:\cdots:y_s)\longmapsto (0:\cdots:0:y_0:\cdots:y_s),
\end{cases}$

$\beta :$ the blowing up of $V_m^r\times V_m^s$ along $V_m^{r-1}\times V_m^{s-1}$,

$\pi :$ the quotient map of degree $m$,

$\overline{\psi } :$ the blowing up of $V_m^{r+s}$ along $V_m^{r-1}\amalg V_m^{s-1}$.\\
The action of $\mu_m$ on $V_m^r\times V_m^s$ is defined by

\ \ \ $((x_0:\cdots:x_{r+1}),(y_0:\cdots:y_{s+1}))$

\ \ \ \ \ \ \ \ \ $\mapsto ((x_0:\cdots:x_r:\zeta_mx_{r+1}),(y_0:\cdots:y_s:\zeta_my_{s+1}))$,\\
and this $\mu_m$-action naturally extends to that on $Z^{r,s}_m$. 
\end{thm}

We can prove Corollary \ref{cf} by induction on dimension from the following lemma: 
\begin{lem}
Let $m$ be a positive integer prime to $p$ and $k$ be a finite field with $\zeta_{2m} \in k$. Let $W$ be a projective smooth variety over $k$. If $\T(W/k,i-1)$ and $\T(V_m^1\times V_m^1\times W/k,i)$ hold, then $\T(V_m^2\times W/k,i)$ also holds.
\end{lem}
\begin{proof} 
For a variety $V$, put $\CH^i(V)_{\mb{Q}_{\ell}}:=\CH^i(V)\otimes _{\mb{Z}}\mb{Q}_{\ell}$. Let $X:=V_m^2\times W$. By Theorem \ref{isf}, we have the diagram
{\small
\begin{align*}
\xymatrix{
& Z_m^{1,1}\times W \ar[d]^{\beta } \ar[r]^{\pi } &Z_m^{1,1}/\mu_m\times W \ar[d]^{\overline{\psi} }\\
V_m^0\times V_m^0\times W \ar@{^{(}->}[r]^{j}&V_m^1\times V_m^1\times W \ar@{.>}[r]^{ \ \ \varphi } &X & (V_m^0\amalg V_m^0)\times W \ar@{_{(}->}[l]_{i \ \ \ \ \ }.
}
\end{align*}}
Here the above maps are similar to that in Theorem \ref{isf}. From this diagram, we have the following isomorphisms (cf.\ \cite[$\S$2]{SK}, \cite[Example 1.7.6]{Fu})
\begin{align}\nonumber
H^{2i}(\overline{X},\mb{Q}_{\ell}(i))&\oplus H^{2(i-1)}(\overline{W},\mb{Q}_{\ell}(i-1))^{\oplus 2m} \\\label{i1}
\simeq & H^{2i}(\overline{V_m^1\times V_m^1\times W},\mb{Q}_{\ell}(i))^{\mu_m}\oplus H^{2(i-1)}(\overline{W},\mb{Q}_{\ell}(i-1))^{\oplus m^2}
\\\label{i2}
\CH^i(X)_{\mb{Q}_{\ell}}\oplus \CH^{i-1}&(W)_{\mb{Q}_{\ell}}^{\oplus 2m} \simeq \CH^i(V_m^1\times V_m^1\times W)^{\mu_m}_{\mb{Q}_{\ell}}\oplus \CH^{i-1}(W)_{\mb{Q}_{\ell}}^{\oplus m^2}.
\end{align}
Here $H^{2i}(\overline{V_m^1\times V_m^1\times W},\mb{Q}_{\ell}(i))^{\mu_m}$ and $\CH^i(V_m^1\times V_m^1\times W)^{\mu_m}_{\mb{Q}_{\ell}}$ are the $\mu_m$-invariant subspace of $H^{2i}(\overline{V_m^1\times V_m^1\times W},\mb{Q}_{\ell}(i))$, $\CH^i(V_m^1\times V_m^1\times W)_{\mb{Q}_{\ell}}$ respectively. The $G_k$-action and the $\mu_m$-action commute, and the isomorphisms \eqref{i1}\,\eqref{i2} are compatible with the cycle class map \eqref{map}. Therefore if $\T(W/k,i-1)$ and $\T(V_m^1\times V_m^1\times W/k,i)$ hold, then $\T(V_m^2\times W/k,i)$ also holds.
\end{proof}
%%%%%%%%%%%%%%%%%%%%%%%%%%%%%%%%%%%
%%%%%%%%%%%%%%%%%%%%%%%%%%%%%%%%%%%
%%%%%%%%%%%%%%%%%%%%%%%%%%%%%%%%%%%
\section{A zeta value of products of four curves}
Let $X$ be a projective smooth fourfold over a finite field. In his paper \cite{Ko} Theorem 2, Kohmoto gave a formula of the special value of the zeta function of $X$ at $s=2$ assuming the Tate-Beilinson conjecture $\boldsymbol{TB}^2(X)$ in the sense of \cite{Ko} and that $\CH^2(X)$ is finitely generated. His work is based on formulae of zeta values established by Bayer, Neukirch, Schneider, Milne and Kahn (\cite{BN}, \cite{Sc}, \cite{Mi} and \cite{Ka}). Now let $X/k$ be as in Theorem \ref{t}\,(1)\,(b). The Galois group $G_k$ of $\overline{k}/k$ acts semisimply on $H^i(\overline{X},\mb{Q}_{\ell})$ by a theorem of Tate \cite{Ta2}. Hence $\boldsymbol{TB}^2(X)$ holds for $X/k$ by Theorem \ref{t}\,(1)\,(b) and Corollary \ref{ct}\,(1)\,(4), and we obtain the following corollaries using the result of Kohmoto mentioned above. This means that we give a new example which satisfy the assumptions above. 

We introduce some notation. Let $W_{\nu }\Omega _{X,{\rm log}}^n$ be the $\acute {\text{e}}$tale sheaf of the logarithmic part of the Hodge-Witt sheaf $W_{\nu }\Omega _{X}^n$ (\cite{Il}). We define $\mb{Z}/m(2)$  and $\mb{Q}/\mb{Z}(2)$ as follows:
\begin{align*}
\mb{Z}/m(2):=&\left\{ \begin{array}{ll}
\mu_m^{\otimes 2} & \text{if } (m,p)=1 \\
\mu_{m^{\prime }}^{\otimes 2}\oplus W_{\nu }\Omega _{X,{\rm log}}^2[-2] &\text{if } m=p^{\nu}m^{\prime } \text{ for }\nu>0 \text{ and }(m^{\prime },p)=1
\end{array}\right.\\
\mb{Q}/\mb{Z}(2):=& \displaystyle \varinjlim _m \mb{Z}/m(2).
\end{align*}
We define $\mb{Z}(2)$ as the $\acute{\text{e}}$tale sheafification on $X$ of the presheaf of cochain complexes
\begin{align*}
U\longmapsto z^2(U,*)[-4],
\end{align*}
where $z^2(U,*)$ denotes Bloch's cycle complex (\cite{Bl})
\begin{eqnarray*}
\cdots \longrightarrow \ z^2(U,r) \ \stackrel{d_r}{\longrightarrow } \ z^2(U,r-1) \ \stackrel{d_{r-1}}{\longrightarrow } \cdots \stackrel{d_1}{\longrightarrow } \ z^2(U,0).
\end{eqnarray*}
Here $z^2(U,r)$ is placed in degree $-r$. We define the unramified cohomology group $H_{ur}^3(k(X),\mb{Q}/\mb{Z}(2))$ as the kernel of the boundary map of a localization sequence
\begin{align*}
H^3({\rm Spec}(k(X)),\mb{Q}/\mb{Z}(2))\longrightarrow \bigoplus_{x \in X^{(1)}}H_x^4({\rm Spec}(\ml{O}_{X,x}),\mb{Q}/\mb{Z}(2)),
\end{align*}
where $k(X)$ is the function field of $X$, and $X^{(1)}$ is the set of points of $X$ of codimension one.
\begin{rmk}
By Bloch \cite{Bl} and Geisser$-$Levine \cite{GL}, we know that the complex $\mb{Z}(2)\otimes\mb{Z}/m$ is isomorphic to $\mb{Z}/m(2)$ defined above in the derived category of complexes of \'etale sheaves.
\end{rmk}

By Theorem \ref{t} and results of Kohmoto \cite[Theorem 1, Theorem 2]{Ko}, we obtain the following corollaries:
\begin{cor}\label{ft}
Let $X/k$ be as in Theorem \ref{t}\,(1)\,(b) and assume $\dim X=4$.\\
\textup{(1)} The unramified cohomology group $H_{ur}^3(k(X),\mb{Q}/\mb{Z}(2))$ is finite.\\
\textup{(2)} $H^5(X,\mb{Z}(2))$ is finite.\\
\textup{(3)} The intersection pairing
\begin{align*}
\CH^2(X)\times \CH^2(X)\longrightarrow \CH^4(X)\simeq \CH_0(X)\stackrel{\rm{deg}}{\longrightarrow }\mb{Z}
\end{align*}
is non-degenerate when tensored with $\mb{Q}$.\\
\textup{(4)} The subgroup $\CH^2(X,i)_{tor}$ of torsion elements of the higher Chow group $\CH^2(X,i)$ is finite for $i=1,2,3,$ and zero for $i\geq 4$. 
\end{cor}

\begin{rmk}
Corollary \ref{ft}\,(4) holds for arbitrary projective smooth fourfolds over finite fields by \cite[Theorem 1(d)]{Ko}
\end{rmk}

For a smooth projective variety $X$ over a finite field $k=\mb{F}_q$, the zeta function of $X$ is defined as follows:
\begin{align*}
\zeta (X,s):=Z(X/k,q^{-s}) \ \text{ with } \ Z(X/k,t):={\rm exp}\Bigl( \sum_{n\geq 1} \frac{\sharp X(\mb{F}_{q^n})}{n} t^n\Bigr),
\end{align*}
where $\sharp $ denotes the cardinality of a finite set. We define the special value of $\zeta (X,s)$ at $s=2$ as
\begin{align*}
\zeta (X,2)^{*}:= \lim _{s\rightarrow 2} \zeta (X,s)(1-q^{2-s})^{-\rho _2},
\end{align*}
where $\rho _2$ is the order of $\zeta (X,s)$ at $s=2$. 
\begin{cor}
Let $X/k$ be as in Theorem \ref{t}\,(1)\,(b) and assume $\dim X=4$. Then the following formula holds{\rm :}
\begin{align*}
&\zeta (X,2)^{*}=\\
&(-1)^{S(2)}\cdot q^{\chi (X,\ml{O}_X,2)}\cdot \frac{\mid H_{ur}^3(k(X),\mb{Q}/\mb{Z}(2))\mid ^2}{\mid H^5(X,\mb{Z}(2))\mid\cdot R_1 }\cdot \prod_{i=0}^3\mid \CH^2(X,i)_{tor}\mid ^{2\cdot (-1)^i}. 
\end{align*}
Here $R_1$ is the order of the cokernel of the map
\begin{align*}
\CH^2(X)\longrightarrow {\rm Hom}(\CH^2(X),\mb{Z})
\end{align*}
induced by the intersection pairing, $S(2)$ and $\chi (X,\ml{O}_X,2)$ are defined as
\begin{align*}
&S(2) :=\sum_{a>4}\rho _{\frac{a}{2}} \ \ \ \bigl( \rho _{\frac{a}{2}}:=\rm{ord}_{s=\frac{a}{2}}\zeta (X,s)\bigr),\\
\chi (X,\ml{O}_X,2):=\sum_{i,j}&(-1)^{i+j}(2-i)\dim_kH^j_{Zar}(X,\Omega _X^i) \ (0\leq i\leq 2, 0\leq j \leq 4).
\end{align*}
\end{cor}\bigskip

\noindent\textbf{Acknowledgements} The author expresses his gratitude to Professors Kanetomo Sato and Thomas Geisser for many helpful suggestions and comments. He also thank Professors Takao Yamazaki and Noriyuki Otsubo for valuable comments.

\bigskip
\begin{flushleft}
Rin Sugiyama\\
Graduate School of Mathematics, Nagoya University\\
Furo-cho, Chikusa-ku, Nagoya 464-8602, Japan\\
e-mail: rin-sugiyama@math.nagoya-u.ac.jp
\end{flushleft}
\end{document}